\documentclass[12pt,reqno]{amsart}
\usepackage[top=1in, bottom=1in, left=1in, right=1in]{geometry}
\usepackage{times, amsthm, amssymb, amsmath, amsfonts, amsthm, bm, graphicx, mathrsfs}
\usepackage[all]{xy}
\usepackage[usenames,dvipsnames]{color}
\usepackage[pagebackref=true,pdftex]{hyperref}

\newcommand{\excise}[1]{}

\newtheorem{theorem}{Theorem}[section]
\newtheorem{lemma}[theorem]{Lemma}

\newtheorem{cor}[theorem]{Corollary}
\newtheorem{prop}[theorem]{Proposition}

\theoremstyle{definition}
\newtheorem{example}[theorem]{Example}
\newtheorem{remark}[theorem]{Remark}

\newtheorem{definition}[theorem]{Definition}

        {\begin{list}
                {\noindent\makebox[0mm][r]{\arabic{enumi}.}}
                {\leftmargin=5.5ex \usecounter{enumi}}
        }
        {\end{list}}

        {\begin{list}
                {\noindent\makebox[0mm][r]{(\roman{enumi})}}
                {\leftmargin=5.5ex \usecounter{enumi}}
        }
        {\end{list}}

\DeclareMathAlphabet{\mathpzc}{OT1}{pzc}{m}{it}
\usepackage{calligra}
\DeclareMathAlphabet{\mathcalligra}{T1}{calligra}{m}{n}

\def\<{\langle}
\def\>{\rangle}
\def\0{\mathbf{0}}

\def\CC{{\mathbb C}}
\def\cC{{\mathcal C}}
\def\cW{{\mathcal W}}
\def\cU{{\mathcal U}}
\def\DD{{\mathcal D}}

\def\NN{{\mathbb N}}
\def\cO{{\mathcal O}}

\def\RR{{\mathbb R}}

\def\ZZ{{\mathbb Z}}

\def\cN{{\mathcal N}}

\renewcommand\Re{{\mathrm {Re} }}

\def\pos{\operatorname{pos}}

\def\ini{{\operatorname{in}}}
\def\nsupp{\operatorname{nsupp}}

\def\vol{{\rm \operatorname{vol}}}

\def\rank{{\operatorname{rank}}}

\def\nothing{\varnothing}

\numberwithin{equation}{section}
\begin{document}

\mbox{}
\title[]{Gevrey and formal Nilsson solutions of $A$-hypergeometric systems}
\author{Mar\'ia-Cruz Fern\'andez-Fern\'andez}
\address{Departamento de \'Algebra\\
Universidad de Sevilla, Av. Reina Mercedes S/N 41012 Sevilla, Spain.}
\email{mcferfer@algebra.us.es}

\thanks{MCFF was partially supported by MTM2016-75024-P and FEDER}

\subjclass[2010]{13N10, 33C70, 14M25, 32C38.}
\keywords{$A$--hypergeometric system, $D$--module, Gevrey series, formal Nilsson series, initial ideals.}

\begin{abstract}
We prove that the space of Gevrey solutions of an $A$--hypergeometric system along a coordinate subspace is contained in a space of formal Nilsson solutions. Moreover, under some additional conditions, both spaces are equal. 
In the process we prove some other results about formal Nilsson solutions.
\end{abstract}
\maketitle

\mbox{}
\vspace{-15mm}
\parskip=0ex
\parindent2em
\parskip=1ex
\parindent0pt

\setcounter{section}{0}
\section*{Introduction}

We study formal solutions of $A$--hypergeometric systems, also known as GKZ--systems, as they were introduced by Gel'fand, Graev, Kapranov and Zelevinsky (see \cite{GGZ} and \cite{GKZ}). They are systems of linear partial differential equations associated with a pair $(A,\beta)$ where $A$ is a full rank $d\times n$ matrix $A=(a_{ij})=(a_1 \cdots a_n)$ with $a_j\in \ZZ^d$ for all $j=1,\ldots,n$ and $\beta\in \CC^d$ is a vector of complex parameters. Recall that the toric ideal of $A$ is defined as $$I_A:=\langle \partial^{u_+}-\partial^{u_{-}} |\; u\in \ZZ^n , \; Au=0 \rangle\subseteq \CC [\partial_1,\ldots ,\partial_n]$$ where  $(u_+)_j=\max\{u_j,0\}$ and $(u_-)_j=\max\{-u_j,0\}$ for $j=1,\ldots,n$.
The $A$--hypergeometric system $H_A (\beta )$ is the left ideal of the Weyl algebra $D=\CC [x_1 ,\ldots ,x_n ]\langle
\partial_1 ,\ldots ,\partial_n \rangle$ generated by $I_A$ and by the Euler operators $
E_i - \beta_i := \sum_{j=1} ^n a_{ij}x_j\partial_j -\beta_i$ for $i=1,\ldots, d$. 
The $A$--hypergeometric $D$-module is nothing but the quotient $M_A(\beta):=D/H_A(\beta)$.
It is well known that $M_A(\beta)$ is a holonomic $D$-module, see \cite{adolphson} and \cite{GKZ}. In \emph{loc. cit.}, it was also proved that the holonomic rank of $M_A(\beta)$, i.e. the dimension of the space of its holomorphic solutions at a nonsingular point, equals the \emph{normalized volume} of $A$, denoted by $\vol (A)$ (see \ref{nvol}), when $\beta$ is generic. Moreover, $M_A (\beta)$ is regular holonomic if and only if $I_A$ is homogeneous (equivalently, $(1,\ldots,1)$ lies in the rowspan of the matrix $A$), see \cite[Ch. II, 6.2, Thm.]{hotta}, \cite[Thm. 2.4.11]{SST} and \cite[Corollary 3.16]{slopes}.

Gevrey series solutions of a holonomic $D$--module $M$ along a variety $Y$ are closely related with the irregularity sheaf of $M$ along $Y$ defined by Mebkhout \cite{Mebkhout} and with the so called slopes of $M$ along $Y$, see \cite{LM}. The slopes of $M_A(\beta)$ along a coordinate subspace $Y$ were computed in \cite{slopes}.  The spaces of Gevrey series solutions of $M_A(\beta)$ along $Y$ were described in \cite{Fer10} (see also \cite{FC1}, \cite{FC2}) for generic enough parameters $\beta\in\CC^d$.

On the other hand, there is an algorithm that computes, for any regular holonomic left $D$--ideal $I$ and a generic vector $w\in \RR^n$, a set of \emph{canonical series solutions} of $I$ that  belong to certain \emph{Nilsson ring}. These series converge in a certain open set that depends on $w$ and form a basis of holomorphic solutions of $I,$ see \cite[chapters 2.5 and 2.6]{SST}. In \cite{DMM12} the authors introduced a notion of formal Nilsson solutions of $H_A(\beta)$ in the direction of $w$, denoted by $\cN_w (H_A(\beta))$, and they used it to generalize various results in \cite{SST} to the case when $H_A(\beta)$ is not necessarily regular. 

In the papers \cite{Fer10} and \cite{DMM12}, some of the results assume $\beta$ to be (very) generic, meaning that it lies outside a certain infinite (but locally finite) collection of affine hyperplanes. In particular, this condition is stronger than $\beta$ being not \emph{rank jumping}, a condition that only requires to avoid a concrete finite affine subspace arrangement of codimension at least two \cite{MMW}. The set of rank jumping parameters is $\varepsilon(A):=\{\beta\in\CC^d|\; \rank (M_A(\beta))>\vol(A)\}$ and it was computed in \cite{MMW} in terms of the local cohomology modules of the toric ring $S_A=\CC [\partial]/I_A$. In particular they proved that $\varepsilon(A)=\nothing$ if and only if $S_A$ is Cohen--Macaulay.

In this note we prove, for all $\beta\in \CC^d$, that the space of Gevrey series solutions of $M_A(\beta)$ along a coordinate subspace is contained in the space of formal Nilsson solutions of $H_A(\beta)$ in a certain direction, see Theorem \ref{thm: Gevrey-subset-Nilsson}. We also prove that under one additional condition both spaces coincide and that for $\beta\notin\varepsilon(A)$ the dimension of this space is the normalized volume of certain submatrix of $A$, see Theorem \ref{thm: Gevrey-equal-Nilsson}. Moreover, in Section \ref{section-Remarks}, we provide some additional results about formal Nilsson solutions of $H_A(\beta)$. 

\subsection*{Acknowledgments}
I am grateful to Christine Berkesch and Laura Felicia Matusevich for helpful conversations related to this work. I also thank two anonymous referees for useful comments and suggestions that improved the final version of this paper.

\section{Preliminaries.}\label{Preliminaries}

\subsection{Notations.}\label{subsection Notations} Let $A=(a_1 \cdots a_n)$ be a $d\times n$ matrix with columns $a_j\in\ZZ^d$ such that $\ZZ A:=\sum_{j=1}^n \ZZ a_j= \ZZ^d$. 

For a given subset $\tau\subseteq \{1,\ldots ,n\}$, set $\overline{\tau}:=\{1,\ldots ,n\}\setminus\tau$. We shall identify $\tau$ with the set of columns of $A$ indexed by $\tau$ and write $A_{\tau}$ for the submatrix of $A$ with column set $\tau$. We denote by $\Delta_\tau$ the convex hull in $\RR^d$ of all the columns of $A_\tau$ and the origin. We also denote $\pos (\tau):=\sum_{j\in \tau} \RR_{\geq 0} a_j$.

We say that a subset $\sigma\subseteq \{1,\ldots ,n\}$ is a \emph{maximal simplex} if $A_\sigma$ is an invertible matrix. We associate to a maximal simplex $\sigma$ an $n\times (n-d)$ matrix $B_\sigma$, where its columns   are indexed by $\overline{\sigma}$, the $j$--th column of $B_\sigma$ has $\sigma$--coordinates equal to $-A_{\sigma}^{-1}a_j$,  $j$--coordinate equal to one and the rest of coordinates equal to zero. In particular, the columns of $B_\sigma$ form a basis of the kernel of $A$.

For example, if $\sigma=\{1,\ldots,d\}$ then 
$$B_\sigma =\left(\begin{array}{cccc}
               -A_{\sigma}^{-1} a_{d+1} & -A_{\sigma}^{-1}a_{d+2} & \cdots & -A_{\sigma}^{-1}a_{n} \\
               1 & 0 &  & 0 \\
               0 & 1 &  & 0 \\
          \vdots &   & \ddots & \vdots \\
               0 & 0 &   & 1
             \end{array}\right).$$

Recall that for any subset $\tau \subseteq \{1,\ldots ,n\}$, the \emph{normalized volume} of $A_\tau$ (with respect to the lattice $\ZZ^d$) is given by:
\begin{equation}
\vol (A_\tau)
:=d! \vol_{\RR^d}(\Delta_{\tau})\label{nvol}\end{equation} where
$\vol_{\RR^d}(\Delta_{\tau})$ denotes the Euclidean volume
of $\Delta_{\tau}\subseteq \RR^d$. We also denote $\vol (\tau):=\vol (A_\tau)$. If $\sigma\subseteq \{1,\ldots ,n\}$ is a maximal  simplex, then $\vol (\sigma)=[\ZZ^d : \ZZ A_{\sigma}]=|\det (A_{\sigma})|$.

\subsection{Regular triangulations.}  
A vector $w\in \RR^n$ defines an abstract polyhedral complex
$T_{w}$ with vertex set contained in $\{1, \ldots , n\}$ as follows: $\tau \in T_{w}$ iff there
exists a vector $\mathbf{c}\in \RR^d$ such that
\begin{equation}
\langle \mathbf{c} , a_j \rangle =w_j  \mbox{ for all } j\in \tau, \label{equality-subdivision}
\end{equation}
\begin{equation}
\langle \mathbf{c} , a_j \rangle < w_j \mbox{ for all } j \notin \tau \label{inequality-subdivision}.
\end{equation}  
Such a polyhedral complex is called a \emph{regular subdivision} of $A$ if it satisfies $\pos (A)=\cup_{\tau \in T_{w}} \pos (\tau).$ This happens for example if $w\in \RR_{>0}^n$ or if $A$ is \emph{pointed}, i.e. the intersection of $\RR^n_{>0}$ with the rowspan of $A$ is nonempty.

An element $\tau\in T_w$ is called a \emph{facet} of $T_w$ if the rank of $A_\tau$ is $d$. Any regular subdivision is determined by its facets and from now on we will write only $\tau \in T_w$ when $\tau$ is a facet of $T_w$. We say that a regular subdivision $T_w$ of $A$ is a \emph{regular triangulation} of $A$ if all its facets are simplices.

An important case of regular subdivision of $A$ is the following. If $w_j=1$ for all $j=1,\ldots,n$, a facet of $T_w$ is the same as a facet of $\Delta_A$ not containing the origin. This particular regular subdivision of $A$ is denoted by $\Gamma_A$.

Notice that for a maximal simplex $\sigma$, it is straightforward from \eqref{equality-subdivision} and \eqref{inequality-subdivision} that 
\begin{equation}\label{w-Bsigma}
\sigma \in T_w \Longleftrightarrow w B_\sigma >0. 
\end{equation} 

If $T$ is any regular triangulation of $A$ then the set 
\begin{equation}\label{eqn:cone-secondary-fan}
C(T):=\{ w \in \RR^n | \; T=T_{w}\}=\{ w \in \RR^n |\; w B_\sigma >0, \; \forall \sigma \in T\} 
\end{equation} is an open nonempty convex rational polyhedral cone. 
The closures of these cones and their faces form the so called \emph{secondary fan} of $A$, introduced and studied by Gel'fand, Kapranov and Zelevinsky \cite[Chapter 7]{GKZ discrim}. 
When $A$ is pointed, it is easy to see that the secondary fan is a complete fan, i.e. its support is $\RR^n$. In general, it is not necessarily complete but its support contains the orthant $\RR_{\geq 0}^n$. 

\subsection{The $A$--hypergeometric fan.} A vector $w\in \RR^n$ defines a partial order on the monomials of the Weyl Algebra $D$ (and also on the monomials in $\CC[\partial_1,\ldots,\partial_n]$) by defining the $(-w,w)$--\emph{weight} of $x^\alpha \partial^\gamma\in D$ as the real value $\langle w , \gamma-\alpha\rangle$.

The \emph{initial form} of an element $P=\sum_{\alpha,\gamma \in\NN^n} c_{\alpha,\gamma} x^\alpha \partial^\gamma\in D$ with respect to $(-w,w)$, denoted by $\ini_{(-w,w)}(P)$, is the sum of the terms $c_{\alpha,\gamma} x^\alpha \partial^\gamma$, with $c_{\alpha,\gamma}\neq 0$, whose $(-w,w)$--weight is maximum. If $I$ is a left $D$--ideal, its \emph{initial ideal} with respect to $(-w,w)$ is defined as $$\ini_{(-w,w)}(I):=\langle \ini_{(-w,w)}(P)|\; P\in I, \; P\neq 0\rangle.$$

\begin{remark}\label{remark-pointed}
When $A$ is pointed, there is a vector $w'\in \RR_{>0}^n$ in the rowspan of $A$ and we have that $\ini_{w'}(I_A)=I_A$ and $\ini_{(-w',w')}(H_A(\beta))=H_A(\beta)$. Then, for any $w\in \RR^n$ we have that 
$w'':=w'+\epsilon w\in \RR_{>0}^n$, $\ini_w(I_A)=\ini_w(\ini_{w'}(I_A))=\ini_{w''}(I_A)$ and $\ini_{(-w,w)}(H_A(\beta))=\ini_{(-w,w)}(\ini_{(-w',w')}(H_A(\beta)))=\ini_{(-w'', w'')}(H_A(\beta))$ for $\epsilon>0$ small enough, see \cite[Lemma 2.1.6]{SST}. 
\end{remark}

The Gr\"obner fan of $I_A$, see \cite[p. 13]{Sturm} (resp. the small Gr\"obner fan of $H_A(\beta)$, see \cite[p. 60]{SST}) is a rational polyhedral fan in $\RR^n$ whose cones $\cC$ satisfy that $\ini_w (I_A)=\ini_{w'}(I_A)$ (resp. $\ini_{(-w,w)}(H_A(\beta))=\ini_{(-w',w')}(H_A(\beta))$) for all $w,w'\in \mathring{\cC}$, where $\mathring{\cC}$ denotes the relative interior of $\cC$. By Remark \ref{remark-pointed}, these two fans are also complete fans when $A$ is pointed.

\begin{definition}
The $A$--\emph{hypergeometric fan} (at $\beta$) is the coarsest rational polyhedral fan in $\RR^n$ that refines both the Gr\"obner fan of $I_A$ and the small Gr\"obner fan of $H_A(\beta)$.
\end{definition}

\begin{remark}\label{rem:hyper-refine-secondary}
We notice that this fan is a refinement of the hypergeometric fan defined in \cite[Section 3.3]{SST} when $I_A$ is homogeneous and $\beta$ is generic. By \cite[Proposition 8.15]{Sturm} and \cite[Corollary 4.4]{Ber-Fer} the $A$--hypergeometric fan is a refinement of the secondary fan of $A$.
\end{remark}

\subsection{$\Gamma$--series.}
Let us denote $\ker_\ZZ (A):=\{u \in \ZZ^n | \; A u=0\}$. Following \cite{GKZ}, for any vector $v\in \CC^n$ such that $A v=\beta$, we consider the $\Gamma$--series $$\varphi_v:=\sum_{u\in \ker_{\ZZ}(A)} \dfrac{x^{v+u}}{\Gamma (v+u+1)},$$ where $\Gamma (v+u+1)=\prod_{j=1}^n \Gamma (v_j+u_j+1)$ and $\Gamma$ is the Euler Gamma function. These series are formally annihilated by $H_A(\beta)$. Moreover, when $I_A$ is homogeneous and $\beta\in\CC^d$ is generic, a basis of convergent $\Gamma$--series solutions of $M_A(\beta)$ can be constructed by using any regular triangulation of $A$, see loc. cit. These $\Gamma$--series are handled in \cite[Section 3.4]{SST} in the following way:
\begin{equation}
\phi_v:=\sum_{u\in N_v} \dfrac{[v]_{u_{-}}}{[v+u]_{u_+}} x^{v+u} \label{eqn:SST-series}
\end{equation} where $N_v=\{u\in \ker_{\ZZ}(A)|\; \forall j=1,\ldots, n,\; v_j+u_j \in \ZZ_{<0} \mbox{ iff } v_j \in \ZZ_{<0} \}$ and $$[v]_{u}=\prod_{j=1}^n v_j (v_j-1)\cdots (v_j-u_j+1).$$

Set $\nsupp (v):=\{j\in \{1,\ldots ,n\}|\; v_j \in \ZZ_{<0}\}$ for any $v\in \CC^n$. 

The series $\phi_v$ is annihilated by $H_A(\beta)$ if and only if $v$ has \emph{minimal negative support}, i.e. there is no $u\in \ker_{\ZZ}(A)$ such that $\nsupp (v+u)\subsetneq \nsupp (v)$, see \cite[Proposition 3.4.13]{SST} whose proof works as well when $I_A$ is not homogeneous.

\begin{remark}\label{remark: phiv-constant-multiple}
It is easy to check that $\Gamma (v+1) \varphi_v = \phi_v$ when $v\in (\CC\setminus \ZZ_{<0})^n$. Notice that $\varphi_v=\varphi_{v+u}$ for any $u\in \ker_{\ZZ}(A)$. Thus, for $u\in N_v$ there is a nonzero scalar $c\in\CC$ such that $\phi_v=c\cdot\phi_{v+u}$.
\end{remark}

\subsection{Gevrey series solutions of $A$-hypergeometric $D$-modules.}\label{subsection-Gevrey}
In this section we introduce the notion of Gevrey series and we recall some notations and results from \cite{Fer10}.
Let us denote, for a subset $\tau\subseteq\{1,\ldots,n\}$, $Y_{\tau}:=\{x_j=0|\; j\in \overline{\tau}\}$ and $x_\tau$ for the set of variables $x_j$ with $j\in\tau$. We denote by $\cO_X$ the sheaf of holomorphic functions on $X=\CC^n$ and by $\cO_{\widehat{X|Y_\tau}}$ the sheaf of formal series along $Y_\tau$. A germ of $\cO_{\widehat{X|Y_\tau}}$ at $p\in Y_\tau$ can be written as
$$f=\sum_{\alpha\in \NN^{\overline{\tau}}}
f_{\alpha}(x_{\tau} ) x_{\overline{\tau}}^{\alpha} \in
\cO_{\widehat{X|Y_\tau},p}\subseteq \CC \{ x_{\tau} - p_{\tau}\}[[
x_{\overline{\tau}} ]]$$ where $f_{\alpha}(x_{\tau} ) \in
\cO_{Y_\tau}(U)$ for certain nonempty relatively open subset $U \subseteq
Y_\tau$, $p\in U$. A formal series $f=\sum_{\alpha\in \NN^{\overline{\tau}}} f_{\alpha}(x_{\tau} ) x_{\overline{\tau}}^{\alpha} \in \cO_{\widehat{X|Y_\tau},p}$ is said to be \emph{Gevrey} of order $s\in \RR$ along $Y_\tau$ at $p\in Y_\tau$ if the series $$\sum_{\alpha\in \NN^{\overline{\tau}}} \frac{f_{\alpha}(x_{\tau}
)}{(\prod_{j\in \overline{\tau}}\alpha_j !)^{s-1}} x_{\overline{\tau}}^{\alpha}$$ is convergent at
$p$.

Since $M_A(\beta)$ is a holonomic $D$--module, any of its formal solutions along $Y_\tau$ is Gevrey of some order. 
We denote by $Hom_D (M_A(\beta),\cO_{\widehat{X|Y_\tau},p})$ the space of all Gevrey solutions of $M_A(\beta)$ along $Y_\tau$ at $p\in Y_\tau$.

Given a maximal simplex $\sigma$ and a vector $\mathbf{k}=(k_i )_{i\notin \sigma }\in \NN^{\overline{\sigma}}$ we denote by $v_{\sigma}^{\mathbf{k}}\in \CC^n$ the vector 
with $\sigma$--coordinates equal to $A_\sigma^{-1}(\beta-A_{\overline{\sigma}}\mathbf{k})$ and $\overline{\sigma}$--coordinates equal to $\mathbf{k}$. Let $\Omega_{\sigma}\subseteq \NN^{\overline{\sigma}}$ be a set of representatives for the different classes with respect to the following equivalence relation in $\NN^{\overline{\sigma}}$: we say that $\mathbf{k}\sim \mathbf{k'}$ if and only if $A_{\overline{\sigma}}\mathbf{k}-A_{\overline{\sigma}} \mathbf{k'}\in \ZZ A_{\sigma}$. Thus, $\Omega_{\sigma}$ is a set of cardinality $\vol (\sigma )=[\ZZ^d :\ZZ A_\sigma]$. 

When $\beta$ is generic, the space of Gevrey solutions of $M_A(\beta)$ along $Y_\tau$ is explicitly described in \cite{Fer10}. 

\begin{theorem}{\rm \cite[Theorem 6.7 and Remark 6.8]{Fer10}} \label{Fer-10-theorem}
If $T(\tau)$ is a regular triangulation of $A_\tau$ that refines 
$\Gamma_{A_\tau}$ and $\beta \in \CC^d$ is generic enough, the set $\{\phi_{v_\sigma^{\mathbf{k}}}: \; \sigma \in T(\tau), \; \mathbf{k} \in \Omega_{\sigma}\}$ is a basis of the space of Gevrey series solutions of $M_A (\beta)$ along $Y_\tau$ at any point $p$ of a certain nonempty relatively open set $\cW_{T(\tau)}\subseteq Y_\tau$. In particular,
$\dim_{\CC} (Hom_D (M_A(\beta),\cO_{\widehat{X|Y_\tau},p}))=\vol (\tau)$.
\end{theorem}

For more precise statements, including the Gevrey order of these series and its relation with the so called \emph{slopes} of $M_A(\beta)$ along $Y_\tau$, see \cite{Fer10}. The slopes of $M_A(\beta)$ along $Y_\tau$ were described in \cite{slopes}.

Notice that if $\tau=A$ then $Y_\tau=\CC^n$, $\cO_{\widehat{X|Y_\tau}}=\cO_X$, and Theorem \ref{Fer-10-theorem} gives a basis of holomorphic functions of $M_A(\beta)$ at any point of $\cW_{T(\tau)}$ when $\beta\in \CC^d$ is generic. Such a basis was first described in \cite{Ohara-Takayama} (see also \cite{GKZ} when $I_A$ is homogeneous).

The following result is the first part of \cite[Theorem 6.2]{Fer10}.

\begin{theorem}\label{thm: lower-bound-Gevrey}
If $p$ is a generic point of $Y_\tau$ then, for all $\beta \in \CC^d$, $$\dim_{\CC} (Hom_D (M_A (\beta),\cO_{\widehat{X|Y_\tau},p}))\geq \vol (\tau).$$
\end{theorem}

\subsection{Formal Nilsson solutions of $A$-hypergeometric $D$-modules}\label{subsection-Nilsson}

We recall here some definitions and results from \cite{DMM12}, see also \cite{SST} when $I_A$ is homogeneous. In the former paper the authors write the following results in terms of a regular triangulation of the matrix $\rho (A):=(\widetilde{a}_0 \; \widetilde{a}_1\cdots \widetilde{a}_n )$, that is constructed from $A$ by adding a first column of zeroes and then a first row of ones. It follows from the definition of regular subdivision (see \eqref{equality-subdivision} and \eqref{inequality-subdivision}) that for a subset $\sigma \subseteq \{1,\ldots, n\}$, we have that $\sigma\in T_w$ if and only if $\{0\}\cup \sigma \in T_{(0,w)}$.

Given a cone $\cC\subseteq \RR^n$, the \emph{dual cone} of $\cC$, denoted by $\cC^*$, is a closed cone consisting of vectors $u\in \RR^n$ such that $\langle w, u\rangle \geq 0$ for all $w\in \cC$ and all $u\in \cC^*$. If $\cC$ is full dimensional, then the cone $\cC^*$ is \emph{strongly convex} (i.e. it doesn't contain non trivial linear subspaces) and $\langle w, u\rangle> 0$ for all $w\in \mathring{\cC}$ and all nonzero $u\in\cC^*$

We say that $w\in \RR_{>0}^n$ is a \emph{weight vector} (for $H_A(\beta)$) if it belongs to the interior of a full dimensional cone of the $A$--hypergeometric fan.

For a weight vector $w\in \RR^n$ we denote by $\cC_w$ the interior of the (full dimensional) cone in the $A$--hypergeometric fan such that $w\in\cC_w$. Notice that $\cC_w\subseteq C(T_w)$.

\begin{definition}\cite[Definition 2.6]{DMM12}\label{def: formal-Nilsson}
Let $w$ be a weight vector for $H_A(\beta)$. Write $\log (x) =(\log x_1 ,\ldots, \log x_n)$. A \emph{basic Nilsson solution of $H_A(\beta)$ in the direction of $w$} is a series of the form 
\begin{equation}\label{eqn: basic-Nilsson}
 \phi=x^v \sum_{u\in C} x^u p_u (\log (x)),
\end{equation} where $v\in \CC^n$, that satisfies 
\begin{enumerate}
 \item[i)] $\phi$ is annihilated by the partial differential operators of $H_A(\beta)$;
 \item[ii)] $C$ is contained in $\ker_{\ZZ}(A)\cap \cC^*$ for some strongly convex open cone $\cC\subseteq \cC_w$ such that $w\in \cC$;
 \item[iii)] the $p_u$ are nonzero polynomials and there exists $K\in \ZZ$ such that $\deg (p_u)\leq K$ for all $u\in C$;
 \item[iv)] $0\in C$.
\end{enumerate} The set $\operatorname{supp}(\phi)=\{v+u |\; u \in C\}$ is called the \emph{support} of $\phi$.

The $\CC$--linear span of all basic Nilsson solutions of $H_A(\beta)$ in the direction of $w$ is called the \emph{space of formal Nilsson solutions of $H_A(\beta)$ in the direction of $w$} and it is denoted by $\cN_w(H_A(\beta))$.
\end{definition}

We define the $w$--weight of a term $x^{v} p (\log(x))$, where $v\in \CC^n$ and $p$ is a polynomial,  as the real value $\Re (\langle w,v\rangle)$. For a series $\phi$ consisting in a (possibly infinite) sum of terms, we say that it has an initial form if there exists the minimum for the set of  $w$--weights of all its nonzero terms. In this case, its initial form in the direction of $w$, denoted by $\ini_w(\phi)$, consists in the sum of all the terms of $\phi$ with minimum $w$--weight.

\begin{remark}\label{remark: initial-form-basic}
Notice that if a series $\phi$ as in \eqref{eqn: basic-Nilsson} satisfies all conditions in Definition \ref{def: formal-Nilsson} we have $\ini_w(\phi)=x^v p_0 (\log (x))$. 
\end{remark}

\begin{prop}\cite[Proposition 2.11]{DMM12}\label{prop: Nilsson-upper-bound}
Let $w\in\RR^n$ be a weight vector for $H_A(\beta)$, then 
$\dim_\CC (\cN_w(H_A(\beta)))\leq \rank (\ini_{(-w,w)}(H_A(\beta)))$.
\end{prop}

We recall that a vector $v\in\CC^n$ is called an \emph{exponent} of $H_A(\beta)$ with respect to $w$ if $x^v$ is a solution of $\ini_{(-w,w)}(H_A(\beta))$.

\begin{theorem}{\rm \cite[Theorem 4.8]{DMM12}}\label{thm: generic-Nilsson-basis}
If $\beta$ is generic and $w$ is a weight vector for $H_A(\beta)$ then the set 
$$\{ \phi_v |\; v \mbox{ is an exponent of } H_A(\beta) \mbox{ with respect to } w\}$$ is a basis of $\cN_w (H_A (\beta))$, where $\phi_v$ is defined in \eqref{eqn:SST-series}.
\end{theorem}

\begin{theorem}\cite[Corollaries 4.9 and 4.11]{DMM12}\label{thm-DMM:generic-case}
If $\beta$ is generic and $w$ is a weight vector for $H_A(\beta)$ then
 \begin{equation}
\dim_{\CC}(\cN_w (H_A(\beta))=\rank (\operatorname{in}_{(-w,w)}(H_A(\beta)))=\deg (\operatorname{in}_w (I_A))=\sum_{\sigma \in T_w} \vol (\sigma). \label{equality-inw-Nilsson}
\end{equation}
\end{theorem}

Let $w\in\RR^n_{>0}$ be a weight vector for $H_A(\beta)$. We say that $w$ is a \emph{perturbation} of a vector $w_0\in \RR^n$ if there exists a full dimensional cone $\cC$ of the $A$--hypergeometric fan such that $w_0\in \cC$ and $w\in \mathring{\cC}$.

We notice that a weight vector $w$ is a perturbation of $(1,\ldots,1)$ if and only if the regular triangulation $T_w$ is a refinement of $\Gamma_A$. 

\begin{theorem}\cite[Theorem 6.4]{DMM12}\label{thm-DDM:convergent-generic}
Assume that $A$ is pointed. If $w$ is a perturbation of $(1,\ldots,1)$ then, for all $\beta\in\CC^d$,
\begin{equation}
\dim_{\CC}(\cN_w (H_A(\beta)))=\rank(\ini_{(-w,w)}(H_A(\beta))=\rank (M_A(\beta)).
\end{equation}
More precisely, $\cN_w(H_A(\beta))$ is the space of convergent series solutions of $M_A(\beta)$ at any point in a certain nonempty open set $\mathcal{U}_{w}\subseteq\CC^n$.
\end{theorem}

\section{Some remarks on formal Nilsson solutions of $H_A(\beta)$.}\label{section-Remarks}

In this section we provide some additional results about $\cN_w (H_A(\beta))$.

\begin{lemma}\label{lemma: Nilsson-lower-bound}
For any $\beta \in \CC^d$ and any weight vector $w$,
\begin{equation}
\dim_{\CC}(\cN_w (H_A(\beta))\geq \sum_{\sigma \in T_w} \vol (\sigma). \label{eqn: Nilsson-lower-bound}
\end{equation}
\end{lemma}

\begin{proof}
By Theorem \ref{thm-DMM:generic-case} equality holds in \eqref{eqn: Nilsson-lower-bound} when $\beta$ is generic and, by Theorem \ref{thm: generic-Nilsson-basis}, there is a basis of $\cN_w (H_A(\beta))$ that consists of the set of series $\phi_v$, see \eqref{eqn:SST-series}, for $v$ varying in the set of exponents of $H_A(\beta)$ with respect to $w$. In this situation, when $\beta$ is not generic, we can apply the same  procedure as in the proof of \cite[Theorem 3.5.1]{SST} and obtain a set of linearly independent formal Nilsson solutions of $H_A(\beta)$ in the direction of $w$. The cardinality of this set is the rightmost quantity in \eqref{eqn: Nilsson-lower-bound}. 
\end{proof}

\begin{cor}\label{cor: equalities-still-hold}
If $w$ is a weight vector, then \eqref{equality-inw-Nilsson} holds for any $\beta\in\CC^d\setminus\varepsilon(A)$.
\end{cor}
\begin{proof}
By \cite[Theorem 4.28]{slopes} and \cite[Lemma 3.1]{Ber-Fer}, $\rank (\operatorname{in}_{(-w,w)}(H_A(\beta)))$ is constant for $\beta\in \CC^d\setminus \varepsilon (A)$ and hence the second equality in \eqref{equality-inw-Nilsson} holds in this case too. The first equality in \eqref{equality-inw-Nilsson} now follows from Lemma \ref{lemma: Nilsson-lower-bound} and Proposition \ref{prop: Nilsson-upper-bound}.
\end{proof}

The following result states that the basis given in Theorem \ref{thm: generic-Nilsson-basis} only depends, up to multiplication of their elements by nonzero scalars, on the regular triangulation $T_w$ and not on the cone $\cC_w\subseteq C(T_w)$. 

\begin{prop}\label{prop: generic-basic-Nilsson}
If $\beta\in\CC^d$ is generic and $w$ is a weight vector, then the set 
$$\mathcal{B}_w (\beta) :=\{\phi_{v_\sigma^{\mathbf{k}}}: \; \sigma \in T_w , \; \mathbf{k} \in \Omega_{\sigma}\}$$ is a basis of $\cN_w (H_A (\beta))$.
\end{prop}

\begin{proof} 
By the assumption on $\beta$, the difference between two vectors in the set $\{v_\sigma^{\mathbf{k}} |\; \sigma\in T_w,\;\mathbf{k}\in \Omega_\sigma\}$ is not an integer vector. Thus, the series in $\mathcal{B}_w (\beta)$ have pairwise disjoint supports, hence they are linearly independent. Moreover, $\nsupp (v_\sigma^{\mathbf{k}})=\nothing$, which implies that $\phi_{v_\sigma^{\mathbf{k}}}$ is annihilated by $H_A (\beta)$, see \cite[Proposition 3.4.13]{SST}.

On the other hand, the support of $\phi_{v_\sigma^{\mathbf{k}}}$ is the set 
$$\operatorname{supp} (\phi_{v_\sigma^{\mathbf{k}}})= \{ v_\sigma^{\mathbf{k}} + (B_\sigma \mathbf{m})^t |\; \mathbf{m}\in \ZZ^{\overline{\sigma}},\; B_\sigma \mathbf{m}\in \ZZ^n \mbox{ and } \mathbf{k}+\mathbf{m} \in \NN^{\overline{\sigma}}\}.$$
Notice that $v_\sigma^{\mathbf{k}} + (B_\sigma \mathbf{m})^t = v_\sigma^{\bf 0} + (B_\sigma ( {\mathbf{k} + \mathbf{m}}))^t$ where ${\mathbf{k} + \mathbf{m}}\in \NN^n$. Thus, since $w B_\sigma >0$ for any $\sigma \in T_w$, we have that $\Re (\langle w,v'\rangle)\geq \Re(\langle w, v_\sigma^{\bf 0}\rangle)$ for all $v'\in \operatorname{supp} (\phi_{v_\sigma^{\mathbf{k}}})$. It follows that there exists the initial form $\ini_w (\phi_{v_\sigma^{\mathbf{k}}})$ and that it consists in a finite sum of terms. By the proof of \cite[Theorem 2.5.5]{SST} we also have that $\ini_w (\phi_{v_\sigma^{\mathbf{k}}})$ is a solution of $\ini_{(-w,w)}(H_A(\beta))$, but a basis of its solutions is given by the set of monomials $x^v$ for $v$ varying in the set of exponents of $H_A(\beta)$ (see Theorems \ref{thm: generic-Nilsson-basis} and \ref{thm-DMM:generic-case}). It follows that there exists an exponent $\widetilde{v}_\sigma^{\mathbf{k}}$ of $H_A(\beta)$ with respect to $w$ such that $\widetilde{v}_\sigma^{\mathbf{k}}\in \operatorname{supp}(\phi_{v_\sigma^{\mathbf{k}}})$, hence $\phi_{v_\sigma^{\mathbf{k}}}=c_{\sigma,\mathbf{k}} \cdot \phi_{\widetilde{v}_\sigma^{\mathbf{k}}
}$ for some nonzero scalar $c_{\sigma,\mathbf{k}}\in\CC$, see Remark \ref{remark: phiv-constant-multiple}.  
This implies that $\mathcal{B}_w (\beta)$ is a basis of $\cN_w (H_A (\beta))$ in this case, by Theorem \ref{thm: generic-Nilsson-basis}.
\end{proof}

\begin{cor}\label{cor: w-constant-Nilsson} 
If $w, w'$ are weight vectors for $H_A(\beta)$, then we have the following:
\begin{enumerate}
 \item[i)] If $\beta$ is generic, $\cN_w (H_A (\beta))=\cN_{w'} (H_A (\beta))$ if and only if $T_w=T_{w'}$. 
 \item[ii)] For all $\beta\in \CC^d\setminus\varepsilon (A)$, $\cN_w (H_A (\beta))=\cN_{w'} (H_A (\beta))$ if  $\cC_w=\cC_{w'}$.
 \item[iii)] For all $\beta\in \CC^d\setminus\varepsilon (A)$, the cone $\cC$ in Definition \ref{def: formal-Nilsson} can be chosen to be $\cC_w$ for any basic Nilsson solution of $H_A(\beta)$ in the direction of $w$.
\end{enumerate}
\end{cor}
\begin{proof}
Proposition \ref{prop: generic-basic-Nilsson} directly implies i).  Let us prove ii).  
We can take a basis $\mathcal{B}=\{\phi_1,\ldots, \phi_r\}$ of $\cN_w (H_A(\beta))$ so that each $\phi_i=x^{v^{(i)}} \sum_{u\in C_i} p_u^{(i)}(\log(x))$ is a basic Nilsson solution of $H_A(\beta)$ in the direction of $w$. Thus, for $i=1,\ldots,r$, there exists a strongly convex open cone $\cC_i$ as in condition ii) of Definition \ref{def: formal-Nilsson}, i.e. $w\in \cC_i \subseteq \cC_w$ and $C_i\subseteq \cC_i^*\cap \ker_\ZZ (A)$. Then the strongly convex open cone $\cC:=\cap_{k=1}^r \cC_k \subseteq \cC_w$ satisfies that condition for all $i=1,\ldots,r$, since $\cC\subseteq\cC_i$ imples $\cC_i^*\subseteq \cC^*$. It follows that for all $w'\in \cC$, the series $\phi_i$ are also basic Nilsson solutions in the direction of $w'$ and, by Remark \ref{remark: initial-form-basic}, $\ini_{w'}(\phi_i)=x^{v^{(i)}} p_0^{(i)}(\log(x))$. This implies that $\cN_w (H_A (\beta))\subseteq \cN_{w'} (H_A (\beta))$, for all $w'\in \cC$. But this last inclusion must be an equality since both spaces have the same dimension, see Corollary \ref{cor: equalities-still-hold}. It follows that the space $\cN_w (H_A (\beta))$, its subset of basic Nilsson solutions $\phi_i$ and their initial forms $\ini_w (\phi_i)$ are locally constant with respect to the weight vector $w$, hence they are constant in the whole open cone $\cC_w$ of the $A$--hypergeometric fan at $\beta$. This proves ii). For iii), notice that, the fact that $\ini_{w'}(\phi_i)=x^{v^{(i)}} p_0^{(i)}(\log(x))$ for all $w'\in \cC_w$ implies that $\langle w',u\rangle \geq 0$ for all $u\in C_i$ and all $w'\in \cC_w$. Thus, $C_i\subseteq \cC_w^*$ for all $i=1,\ldots,r$, which proves the result.
\end{proof}

\begin{remark}
If $\beta$ is generic and $w,w'$ are weight vectors such that $T_w=T_{w'}$, it may happen that $\ini_w (\phi )\neq \ini_{w'}(\phi)$ for some series $\phi\in \mathcal{B}_w (\beta)=\mathcal{B}_{w'} (\beta)$ (in which case $\ini_{(-w,w)}(H_A(\beta))\neq \ini_{(-w',w')}(H_A(\beta))$). For example, let us consider the matrix 
 $$A=\left(\begin{array}{rrrr}
 1&1&1&1\\
 0&1&2&3\end{array}\right)$$ 
and a generic $\beta\in\mathbb{C}^3$. The maximal simplex $\sigma=\{1,4\}$ defines a regular triangulation of $A$ induced by any of the weight vectors $w^{(1)}=(1,2,5,1)$ and $w^{(2)}=(1,5,2,1)$. We can choose $\Omega_\sigma$ so that the series 
$\phi_v$ with $v=(\beta_1-(\beta_2+2)/3,1,0,(\beta_2-1)/3)$, see \eqref{eqn:SST-series}, belongs to $\mathcal{B}_{w^{(1)}} (\beta)=\mathcal{B}_{w^{(2)}} (\beta)$. Notice that 
$$N_v=\{u=(-(2 m_2 + m_3)/3,m_2,m_3,-(m_2+2 m_3)/3)\in \ZZ^4|\; m_2\geq -1, \; m_3 \geq 0\; \}.$$ Thus, $\ini_{w^{(1)}}(\phi_v)=x^v\neq \ini_{w^{(2)}}(\phi_v) $ since  $v+(0,-1,2,-1)\in \operatorname{supp}(\phi_v)=v+N_v$ has smaller $w^{(2)}$-weight than $v$.
\end{remark}

The following result improves \cite[Corollary 6.9 ]{DMM12}.

\begin{cor}\label{cor:DMM-improvement}
If $w$ is a weight vector, we have, for all $\beta\in\CC^d$, that 
\begin{equation}\label{eqn: convergent-Nilsson-lower-bound}
\dim_\CC (\{\phi \in \cN_w (H_A(\beta)) | \; \phi \mbox{ is convergent}\})\geq \sum_{\sigma \in T_w^0 }\vol (\sigma ) 
\end{equation}
where $T_w^0=\{\sigma \in T_w |\; \sigma \subseteq \eta \mbox{ for some } \eta\in \Gamma_A \}$. Moreover, equality holds in \eqref{eqn: convergent-Nilsson-lower-bound} if $\beta$ is generic.
\end{cor}

\begin{proof}
Assume first that $\beta\in\CC^d$ is generic. Thus, $\mathcal{B}_w (\beta)$ is a basis of $\cN_w (H_A(\beta))$ by Proposition \ref{prop: generic-basic-Nilsson}. On the other hand, it follows from \cite[Theorem 3.11]{Fer10} and the definition of $B_\sigma$, that $\phi_{v_\sigma^{\mathbf{k}}}$ is convergent if and only if 
$(1,\ldots, 1)B_\sigma \geq 0$, that 
holds if and only if $\sigma$ is contained in a facet of $\Gamma_A$. Thus, $\mathcal{B}_w^0(\beta) :=\{\phi_{v_\sigma^{\mathbf{k}}}: \; \sigma \in T_w^0 , \; \mathbf{k} \in \Omega_{\sigma}\}$ is a linearly independent set of convergent formal Nilsson solutions of $H_A(\beta)$ in the direction of $w$. The equality in \eqref{eqn: convergent-Nilsson-lower-bound} for generic $\beta$ follows from the fact that the series in $\mathcal{B}_w (\beta)\setminus \mathcal{B}^0_w(\beta)$ are all divergent and have pairwise disjoint supports, so no linear combination of them can be convergent. If $\beta\in \CC^d$ is not assumed to be generic, we can consider a generic parameter $\beta'$ and apply to the set $\mathcal{B}_w^0 (\beta + \epsilon \beta')$, with $\epsilon \in \CC$ such that $|\epsilon|$ is small enough, the same method as in \cite[Theorem 3.5.1]{SST} to get the desired lower bound in \eqref{eqn: Nilsson-lower-bound}.
\end{proof}

\begin{lemma}{\rm \cite[Lemma 5.3]{Saito-log-free}}\label{lemma: diff-log-term}
Let $p(y)\in \CC [y]$, $v\in\CC^n$ and $\nu\in\NN^n$, then $$\partial^\nu [x^v p (\log(x))]=x^{v-\nu} \sum_{0\leq \nu'\leq \nu} \lambda_{\nu'} \partial^{\nu-\nu'}[p](\log (x))$$ where the sum is over $\nu'\in \NN^n$ such that $\nu_j'\leq \nu_j$ for all $j$, and $\lambda_{\nu'}\in\CC$. In particular, $$\lambda_\nu=[v]_\nu=\prod_{j=1}^nv_j(v_j-1)\cdots (v_j-\nu_j+1).$$
\end{lemma}

The following lemma guarantees that in the third condition of Definition \ref{def: formal-Nilsson} we can assume that the constant $K$ is independent of $\beta$.

\begin{lemma}\label{lemma: deg-p-bound}
Let $w$ be a weight vector. Then for any basic Nilsson solution $\phi$ of $H_A(\beta)$ in the direction of $w$ as in \eqref{eqn: basic-Nilsson} we have that 
$$\deg (p_u)\leq (n+1)(2^{2(d+1)}\vol(A)-1)$$ for all $u\in C$.
\end{lemma}
\begin{proof}
Assume first that $I_A$ is homogeneous. In this case $M_{A_\tau}(\beta-A_{\overline{\tau}}\alpha)$ is regular holonomic \cite{hotta} and $\deg (p_u)\leq n (2^{2 d}\vol(A)-1)$ by \cite[Theorem 2.5.14 and Corollary 4.1.2]{SST}.

If $I_A$ is not homogeneous, we can consider the $(d+1)\times (n+1)$ matrix $\rho(A)$ as defined at the beginning of Subsection \ref{subsection-Nilsson}. 
Notice that $\phi$ is also a basic Nilsson solution of $H_A(\beta)$ in the direction of $w'$ for all $w'\in \mathcal{C}$, where $\cC$ is an open cone as in condition ii) in Definition \ref{def: formal-Nilsson}. In particular, we can assume without loss of genericity that $w$ is generic. This implies that $(0,w)+\lambda (1,\ldots ,1)\in \RR^{n+1}$ is generic if $\lambda>0$ is generic. Since $(1,\ldots,1)$ belongs to the rowspan of $\rho(A)$, we have that $(0,w)$ is a weight vector for $H_{\rho(A)}(\beta_0,\beta)$, where $\beta_0\in\CC$ (see also  \cite[Remark 2.5]{DMM12}). 

Assume that $\beta_0\in\CC$ is sufficiently generic and consider the following series, see \cite[Definition 3.16]{DMM12},
$$\rho(\phi):=\sum_{u\in C}\partial_0^{|u|}[x_0^{\beta_0-|v|}x^{v+u}\widehat{p}_u(\log (x_0),\ldots, \log(x_n))]$$ where $|u|:=\sum_{j=1}^n u_j$, $\partial_0^{-k}$ is defined in \cite[Definition 3.13]{DMM12} when $k>0$, and $\widehat{p}_u\in \CC [y_0,\ldots,y_n]$ is defined from $p_u$ as in \cite[(3.2)]{DMM12}. We remark here that $\deg(p_u)\leq \deg(\widehat{p}_u)$ because $\widehat{p}_u(0,y_1,\ldots, y_n)=p_u(y_1,\ldots,y_n)$. 

By Lemma \ref{lemma: diff-log-term} and \cite[Lemma 3.12 and Definition 3.13]{DMM12}, we can write
$$\rho(\phi)=\sum_{u\in C}x_0^{\beta_0-|v|-|u|}x^{v+u} h_u(\log (x_0),\ldots, \log(x_n))$$ where $h_u(y)\in \CC[y_0,\ldots,y_n]$. 
By Lemma \ref{lemma: diff-log-term}, if $|u|\geq 0$, then $h_u$ equals $[\beta_0-|v|]_{|u|} \cdot \widehat{p}_u$ plus other polynomial of degree smaller than $\deg(\widehat{p}_u)$. This implies that $\deg(h_u)=\deg(\widehat{p}_u)$ when $|u|\geq 0$ because $\beta_0$ is generic. For $|u|<0$ it is also true that $\deg(h_u)=\deg(\widehat{p}_u)$ by \cite[Lemma 3.12 and Definition 3.13]{DMM12}.

On the other hand, by \cite[Proposition 3.17]{DMM12} the series $\rho(\phi)$ is a basic Nilsson solution in the direction of $(0,w)$ of the hypergeometric system $H_{\rho(A)}(\beta_0,\beta)$ and since $I_{\rho(A)}$ is homogeneous, we have that $\deg(h_u)\leq (n+1)(2^{2(d+1)}\vol(\rho(A))-1)$, where $\vol(\rho(A))=\vol(A)$. Thus, $$\deg(p_u)\leq \deg(\widehat{p}_u)=\deg(h_u) \leq (n+1)(2^{2(d+1)}\vol(A)-1).$$
\end{proof}

\section{Gevrey versus formal Nilsson solutions of $H_A(\beta)$.}\label{section-Gevrey-versus-Nilsson}

Let $\tau\subseteq\{1,\ldots, n\}$ be a subset such that $A_\tau$ is pointed and $\rank(A_\tau)=d$. In this section we prove, for all $\beta\in \CC^d$, that the space of Gevrey solutions of $M_A(\beta)$ along a coordinate subspace $Y_\tau$ is contained in the space of formal Nilsson solutions of $H_A(\beta)$ in a certain direction, see Theorem \ref{thm: Gevrey-subset-Nilsson}. If we further assume that $\pos(A)=\pos(A_\tau)$, we also prove that both spaces are the same and, for $\beta\in\CC^d\setminus\varepsilon(A)$, its corresponding dimension is $\vol(\tau)$, see Theorem \ref{thm: Gevrey-equal-Nilsson}.

The following result follows from \cite[Lemma 3.6]{DMM12} (see also \cite[Lemma 4.1.3]{SST}), \cite[(3.2)]{SVW-bounds} (see also \cite[(3.13)]{SST}) and \cite[Corollary 8.4]{Sturm}.

\begin{lemma}\label{lemma: exponent-positive-coordinates}
If $w$ is a weight vector and $v$ is an exponent of $H_A(\beta)$ with respect to $w$, there exists $\sigma\in T_w$ such that $v_j\in\NN$ for all $j\notin \sigma$.
\end{lemma}

\begin{lemma}\label{lemma-triangulations}
For any regular triangulation $T(\tau)$ of $A_\tau$ there exists a regular triangulation $T$ of $A$ such that $T(\tau)\subseteq T$. In particular, if $\pos (A)=\pos (A_\tau)$ then  $T(\tau)=T$.
\end{lemma}
\begin{proof}
By definition of regular triangulation, there is a weight vector $w(\tau) \in \RR^\tau$ such that $T(\tau)=T_{w(\tau)}$. Then choose another generic vector $w(\overline{\tau})\in \RR_{>0}^{\overline{\tau}}$, and consider $w\in\RR^n$ to be a vector with $\tau$--coordinates $\epsilon w(\tau)$, with $\epsilon >0$ small enough, and $\overline{\tau}$--coordinates $w(\overline{\tau})$. Since $\epsilon w(\tau)$, $w(\overline{\tau})$ and $\epsilon>0$ can be chosen to be generic, it follows that $w$ induces a regular triangulation $T:=T_w$ of $A$. By using the definition of regular triangulation, see conditions \eqref{equality-subdivision} and \eqref{inequality-subdivision} (but substitute $\tau$ there by a maximal simplex $\sigma$), it is easy to check that $T(\tau)\subseteq T$ if $\epsilon >0$ is small enough. 
\end{proof}

Let $T(\tau)$ be a regular triangulation of $A_\tau$ refining $\Gamma_{A_\tau}$ and $w$ a weight vector for $H_A(\beta)$ chosen as in the proof of Lemma \ref{lemma-triangulations}. Thus, by the assumption on $T(\tau)$ we have that $w(\tau)$ is a perturbation of $(1,\ldots,1)\in \RR^\tau$. Recall that $w$, $w(\tau)$ and $w(\overline{\tau})$ are all chosen to be generic.

\begin{theorem}\label{thm: Gevrey-subset-Nilsson}
If $A_\tau$ is pointed and $\rank(A_\tau)=d$, any Gevrey solution of $M_A(\beta)$ along $Y_\tau$ (at $p\in \mathcal{U}_{ w (\tau)}$) can be written as a formal Nilsson solution of $H_A(\beta)$ in the direction of $w$.
\end{theorem}

\begin{proof}
Let $f=\sum_{\alpha\in \NN^{\overline{\tau}}}
f_{\alpha}(x_{\tau} ) x_{\overline{\tau}}^{\alpha} \in\cO_{\widehat{X|Y},p}$ be any Gevrey series solution of $M_A(\beta)$. By \cite[Lemma 6.11]{Fer10} we have that $f_{\alpha}(x_{\tau} )$ is a holomorphic solution of $M_{A_\tau}(\beta-A_{\overline{\tau}}\alpha)$ at $p$. Thus, by Theorem \ref{thm-DDM:convergent-generic}, each $f_{\alpha}(x_{\tau})$ can be written as an element of $\cN_{w(\tau)}(H_{A_\tau}(\beta-A_{\overline{\tau}}\alpha))$, i.e. a finite linear combination of series of the form \eqref{eqn: basic-Nilsson} in the variables $x_\tau$ convergent at $p$. In particular, $f$ can be rewritten as a sum of terms $x^{\gamma} q_{\gamma}(\log x_\tau)$ where $\gamma \in \CC^n$ and $q_{\gamma}\in\CC[x_\tau]$ are polynomials with degree $\deg (q_\gamma)\leq (n+1) 2^{2(d+1)} \vol(A_\tau)$, see Lemma \ref{lemma: deg-p-bound}. Moreover, notice that the result of applying a monomial $x^\lambda \partial^\mu$ in the Weyl Algebra $D$ to $x^{\gamma} \log (x)^\nu$ is of the form $x^{\gamma-\mu+\lambda} g(\log (x))$ for some polynomial $g$. 
This implies that any subseries of $f$ of the form $\sum_{\gamma \in (v+\ZZ^n)} x^{\gamma} q_{\gamma}(\log x_\tau)$, for some $v\in \CC^n$, is still annihilated by $H_A(\beta)$ and hence defines a Gevrey solution of $M_A(\beta)$. Since the dimension of the space of Gevrey series solutions is finite, $f$ is a finite sum of such subseries, say $F_1,\ldots, F_r$. Since $F_k$ is annihilated by the Euler operators of $A$ we have that $A\gamma=\beta$ for all $\gamma\in \CC^n$ such that $q_\gamma \neq 0$. Thus, we can write each $F_k$ in the form $$F_k = x^v \sum_{u \in \ker_\ZZ (A)} x^{u} p_{u}(\log (x_{\tau}))$$ where $v\in \CC^n$ and $p_u =q_{v+u}$. It is enough to prove that a Gevrey solution $F_k$ of $M_A(\beta)$ is a formal Nilsson solution of $H_A(\beta)$ in the direction of $w$ and we can assume for simplicity that the original $f=\sum_{\alpha\in \NN^{\overline{\tau}}}
f_{\alpha}(x_{\tau} ) x_{\overline{\tau}}^{\alpha} \in\cO_{\widehat{X|Y},p}$ can also be written in this form. Notice that we have shown that $f$ can be written as in \eqref{eqn: basic-Nilsson} satisfying conditions i) and iii) in Definition \ref{def: formal-Nilsson}. We need to prove that the support of $f$ is of the form $v'+C$ for some $v'\in v+\ker_{\ZZ}(A)$ and $C$ satisfying conditions ii) and iv) in Definition \ref{def: formal-Nilsson}.

Notice that for all $\alpha\in\NN^{\overline{\tau}}$, we have
$$f_{\alpha}(x_{\tau} )= \sum_{u}x^{(v+u)_{\tau}} p_u(\log(x_\tau))$$ where the sum is over the vectors $u \in \ker_{\ZZ}(A)$ such that $(v+u)_{\overline{\tau}}=\alpha$ (in particular, $u_\tau$ varies in a translate of $\ker_{\ZZ}(A_\tau)$).

Let us see that $f$ has an initial form with respect to $w$. It is enough to find a lower bound for the $w$-weights of all terms $x^{v+u}p_u(\log(x_\tau))$ of $f$ with $p_u\neq 0$.

By \cite[Theorem 2.5.5]{SST}, the initial form of $f_\alpha \in \cN_{w(\tau)}(H_{A_\tau}(\beta-A_{\overline{\tau}}\alpha))$ with respect to $w(\tau)$ is a solution of $\ini_{(-w(\tau),w(\tau))}(H_{A_\tau}(\beta-A_{\overline{\tau}}\alpha))$. Since $w(\tau)$ is generic, $\ini_{w(\tau)}(f_\alpha)=x_\tau^{\widetilde{v}_\tau} p_{u'}(\log(x_\tau))$ for some $\widetilde{v}_\tau \in \CC^\tau$ and $u'\in \ker_{\ZZ} (A)$. By \cite[Theorems 2.3.3(2), 2.3.9, 2.3.11]{SST}, $\widetilde{v}_\tau$ is an exponent of $H_{A_\tau}(\beta-A_{\overline{\tau}}\alpha)$ with respect to $w(\tau)$. Thus, by Lemma \ref{lemma: exponent-positive-coordinates} there is a maximal simplex $\sigma\in T_{w(\tau)}$ such that $\widetilde{v}_j\in \NN$ for all $j\in \tau\setminus \sigma$.

Let us denote by $\widetilde{v}\in\CC^n$ the vector with $\overline{\tau}$-coordinates $\widetilde{v}_{\overline{\tau}}=\alpha$ and $\tau$-coordinates equal to $\widetilde{v}_\tau$. Notice that $\widetilde{v}_{\overline{\sigma}}\in \NN^{\overline{\sigma}}$ and since $A \widetilde{v} =\beta$ we have that $\widetilde{v}_{\sigma}=
A_{\sigma}^{-1}(\beta-A_{\overline{\sigma}}\widetilde{v}_{\overline{\sigma}})$. We can write 
$$\widetilde{v}=\beta^\sigma + B_\sigma \widetilde{v}_{\overline{\sigma}}$$ where $B_\sigma$ was defined in Subsection   \ref{subsection Notations} and $\beta^{\sigma}\in \CC^n$ denotes the vector whose $\sigma$-coordinates agree with  $A_{\sigma}^{-1}\beta$ and whose other coordinates are zero.

Moreover, by Lemma \ref{lemma-triangulations}, $\sigma\in T_w$. Thus, by \eqref{w-Bsigma} and the fact that $\widetilde{v}_{\overline{\sigma}}\in\NN^{\overline{\sigma}}$, we have that 
$$\Re(\langle w , \widetilde{v}\rangle)=\Re(\langle w , \beta^\sigma \rangle)+ \langle w , B_\sigma \widetilde{v}_{\overline{\sigma}}\rangle=\Re(\langle w , \beta^\sigma \rangle)+ \langle w B_\sigma, \widetilde{v}_{\overline{\sigma}}\rangle\geq \Re(\langle w , \beta^{\sigma}\rangle)$$  where $\Re(\langle w , \widetilde{v}\rangle)$ is the $w$-weight of $\ini_{w}(f_\alpha (x_\tau) x_{\overline{\tau}}^\alpha)=\ini_{w(\tau)}(f_\alpha (x_\tau)) x_{\overline{\tau}}^\alpha$.

Then the minimum of the finite set 
$$\{\Re(\langle w , \beta^{\sigma}\rangle)|\; \sigma\in T_{w(\tau)}\}$$ is a lower bound for the $w$-weights of all the terms of $f$. Thus, $f$ has an initial form with respect to $w$ and $\ini_w(f)$ must be a term $x^{v+u'}p_{u'}(\log(x_\tau))$ for some $u'\in \ker_{\ZZ}(A)$ because $w$ is generic. We may assume for simplicity that $\ini_w(f)=x^v p_0(\log(x_\tau))$ since $v+\ker_{\ZZ}(A)=v+u'+\ker_{\ZZ}(A)$. This implies that $x^v p_0(\log(x_\tau))$ is a solution of $\ini_{(-w,w)}(H_A(\beta))$ by the same argument as in the proof of \cite[Theorem 2.5.5]{SST}, hence $v$ is an exponent of $H_A(\beta)$ with respect to $w$. 
The set $C:=\{u\in\ker_\ZZ (A)|\; p_u\neq 0\}$ satisfies condition iv) in Definition \ref{def: formal-Nilsson}.

Let $w'\in \cC_w$ be generic and such that $w_\tau '$ is a perturbation of $w(\tau)$. Then the previous argument works as well for $w'$ instead of $w$ and we have that $f$ has also an initial form with respect to $w'$ of the form $\ini_{w'}(f)=x^{v+u'} p_{u'} (\log(x_\tau))$ and that $v+u'$ is an exponent of $H_A(\beta)$ with respect to $w'$. Since the set of exponents of $H_A(\beta)$ with respect to $w'$ is finite and constant for all $w'\in \mathcal{C}_w$, we can find an open cone $\cC'\subseteq \mathcal{C}_w$ such that $w\in \cC'$ and $\ini_{w'}(\phi)=x^v p_0(\log(x_\tau)$ for all $w'\in \cC'$. Hence, we have $\langle w',u\rangle> 0$ for all $w'\in \cC'$ and all $u\in C\setminus\{0\}$. Thus, $f$ satisfies condition ii) in Definition \ref{def: formal-Nilsson}.
\end{proof}

The following result provides a partial converse to Theorem \ref{thm: Gevrey-subset-Nilsson}.

\begin{theorem}\label{thm: Gevrey-equal-Nilsson}
If $A$ is pointed, $\pos (A_\tau)=\pos (A)$ and $p\in \mathcal{U}_{w(\tau)}\subseteq Y_\tau$, we have $$\dim_\CC (Hom_D (M_A(\beta),\cO_{\widehat{X|Y_\tau},p}))=\dim_\CC (\cN_w (H_A(\beta)))$$ for all $\beta\in \CC^d$. More precisely, $\cN_w(H_A(\beta))$ is the space of Gevrey series solutions of $M_A(\beta)$ along $Y_\tau$ at any point $p\in \mathcal{U}_{w(\tau)}$. Moreover, for $\beta\in \CC^d\setminus \varepsilon (A)$, the dimension of this space is $\vol(\tau)$.
\end{theorem}
\begin{proof}
By Theorem \ref{thm: Gevrey-subset-Nilsson}, it is enough to show that any basic Nilsson series $\phi$ in the direction of $w$ as in \eqref{eqn: basic-Nilsson} is a Gevrey series along $Y_\tau$ at any point $p\in \mathcal{U}_{w(\tau)}\subseteq Y$.

Since $\pos (A_\tau)=\pos (A)$ we have that $T(\tau)=T_w$ by the construction of $w$, see the proof of Lemma \ref{lemma-triangulations} and the subsequent paragraph.

By \cite[Lemma 2.10]{DMM12} we have that $v$ is an exponent of $H_A(\beta)$ with respect to $w$. Thus, there exists $\sigma \in T_w=T(\tau)$ such that $v_j\in \NN$ for all $j\notin\sigma$, see Lemma \ref{lemma: exponent-positive-coordinates}. In particular, we have $v_{\overline{\tau}}\in\NN^{\overline{\tau}}$. Moreover, by \cite[Proposition 5.4]{Saito-log-free} (whose proof is also valid when $I_A$ is not necessarily homogeneous) we have that $p_u (y) \in  \CC[ y_j :\;  j \in \operatorname{vert}(T_w )]$ for all $u\in C$, where $\operatorname{vert}(T_w )$ is the set of vertices of $T_w$. In particular, $p_u \in \CC[y_j :\; j\in \tau]$. Thus, we can define  $p_u(y_\tau):=p_u(y)$ and write  
$$\phi=\sum_{u\in C} x^{v+u} p_u(\log(x_\tau))=\sum_{\alpha \in \ZZ^n} f_{\alpha}(x_\tau) x_{\overline{\tau}}^{\alpha}$$ where 
$$f_{\alpha}(x_\tau)=\sum_{u\in C, \; (v+u)_{\overline{\tau}}=\alpha}x^{(v+u)_{\tau}} p_u(\log(x_\tau))$$ is annihilated by $H_{A_\tau}(\beta-A_{\overline{\tau}}\alpha)$, see the proof of \cite[Lemma 6.11]{Fer10}. 

We need to prove that $f_\alpha=0$ if $\alpha\notin \NN^n$. That is, we need to show that $v_j+u_j\geq 0$ for all $j\notin \tau$ and $u\in C$, where $v+C=\operatorname{supp}(\phi)$. Since $\ini_w(\phi)=x^v p_0(\log(x_\tau))$, we know that $\langle w,u\rangle > 0$ for all $u\in C\setminus\{0\}$. Assume to the contrary that there exist $u\in C$ and $j\in\overline{\tau}$ such that $v_j+u_j<0$ and choose such a vector $u$ so that $\langle w,u\rangle$ is minimal. Since $j\in\overline{\tau}$ and $\pos(A)=\pos(A_{\tau})=\cup_{\sigma \in T(\tau)} \pos(\sigma)$, there exist $m_\sigma \in \NN^\sigma$ and $m_j\in \NN$ with $m_j\geq 1$ such that $m_j a_j=\sum_{i\in\sigma} m_i a_i$. Then $u(m):=-m_j e_j+\sum_{i\in\sigma} m_i e_i \in \ker_{\ZZ}(A)$, where $e_\ell$ denotes the $\ell$-th vector of the standard basis of $\RR^n$. 
Notice that $P=\partial^{u(m)_{+}}-\partial^{u(m)_{-}}=\partial_\sigma^{m_\sigma}-\partial_j^{m_j}\in H_A(\beta)$, hence $P(\phi)=0$. 

By Lemma \ref{lemma: diff-log-term} and using that  $\partial_j [p_u]=0$ for all $j\notin \tau$, we have that
$$\partial_j^{m_j} [x^{v+u} p_u(\log(x_\tau))]
=x^{v+u-m_j e_j}\lambda_{m_j} p_u(\log(x_\tau))\neq 0$$ where $\lambda_{m_j}=[v_j+u_j]_{m_j}\neq 0$ because $v_j +u_j<0$ by assumption.
Since $P(\phi)=0$ and $v+u-m_j e_j =v+u+u(m)_{-}$ belongs to the support of the series $\partial^{u(m)_{-}}(\phi)=\partial^{u(m)_{+}}(\phi)$, we have that $v+u+u(m)$ must be in the support of $\phi$, but then $u+u(m)\in C$ and $v_j + u_j +u(m)_j<0$. On the other hand, $\ini_w (P)=-\partial_j^{m_j}$ because $\partial_\sigma^{m_\sigma} \notin \ini_w(I_A)$ for $\sigma \in T(\tau)=T_w$, see \cite[Corollary 8.4]{Sturm}. Thus, $\langle w, u(m)\rangle<0$ and 
we obtain that $\langle w, u+u(m) \rangle <\langle w,u\rangle$ which is in contradiction with our assumption. 

Finally, let us see that each $f_\alpha$ is convergent at $p\in \cU_{w(\tau)}$. Since $\phi$ is basic, there is an open and strongly convex cone $\mathcal{C}\subseteq \cC_w$ such that $w\in \cC$ and $\langle w ',u\rangle >0$ for all $w'\in \mathcal{C}$ and all $u\in C\setminus \{0\}$. In particular, $\ini_{w'}(\phi)=\ini_w(\phi)$ for all $w'\in \mathcal{C}$. This implies that for fixed $\alpha$ and all $w'\in \cC$ with $w_{\overline{\tau}}'=w_{\overline{\tau}}$ there exists the initial form $\ini_{w_\tau'}(f_\alpha)$. Notice that the set of all $w_{\tau}'$ for $w'$ as above is a neighborhood of $w_\tau \in \RR^\tau$ and we can find a smaller neighborhood $U$ of $w_\tau$ such that $\ini_{w_\tau'}(f_\alpha)=\ini_{w(\tau)}(f_\alpha)$ for all $w_\tau'\in U$. Then, by \cite[Remark 2.7]{DMM12}, $f_\alpha\in \cN_{w(\tau)}(H_{A_\tau}(\beta-A_{\overline{\tau}}\alpha))$ and by Theorem \ref{thm-DDM:convergent-generic}, $f_\alpha$ is convergent at any point $p\in \cU_{w(\tau)}$.

Thus, $\phi$ is a formal solution of $M_A(\beta)$ along $Y_\tau$, which implies it is Gevrey of some order because $M_A(\beta)$ is holonomic.

For $\beta\in \CC^d\setminus \varepsilon (A)$, $\dim(\cN_w(H_A(\beta)))=\vol(\tau)$ by Corollary \ref{cor: equalities-still-hold}, where $\cup_{\sigma\in T_w} \Delta_\sigma=\Delta_\tau$ in our case, and Theorem \ref{thm: lower-bound-Gevrey}. 
\end{proof}

\begin{remark}
The additional condition $\pos(A)=\pos(A_\tau)$ in Theorem \ref{thm: Gevrey-equal-Nilsson} is necessary, even if $\beta\in \CC^d$ is generic, as shown by the following example.
\end{remark}

\begin{example}
Let us consider the system $H_A(\beta)$ for  
$$A=\left(\begin{array}{ccc}
   1 & 0 & 3\\
   0 & 1 & -1
  \end{array}\right),$$ and a generic parameter $\beta\in\CC^2$. 
  For $\tau=\{1,2\}$, we have $Y_\tau=\{x_3=0\}$. Let $w=
  (0,0,1)+\epsilon (w(\tau),0)$ where $\epsilon>0$ is small enough and $w(\tau)\in \RR^2$ is a perturbation of $(1,1)$. We have that $T_w =\{\tau ,\sigma\}$ for $\sigma=\{1,3\}$, see Figure 1. Moreover, $T(\tau)=\{\tau\}\subseteq T_w$.
  
  Notice that both simplices $\tau, \sigma$ have normalized volume one. According to Proposition \ref{prop: generic-basic-Nilsson}, $\{\phi_1:=\phi_{v_\tau^0},\phi_2:=\phi_{v_\sigma^0}\}$ is a basis of $\cN_w (H_A(\beta))$, where $v_\tau^0=(\beta_1 ,\beta_2 ,0)$ and $v_\sigma^0=(\beta_1+ 3 \beta_2, 0, -\beta_2)$. Moreover $\ker_{\ZZ}(A)=\ZZ (-3,1,1)$. Thus, $\operatorname{supp}(\phi_1)=\{(\beta_1-3m,\beta_2+m,m)|\; m\in \NN\}$ and $\operatorname{supp}(\phi_2)=\{ (\beta_1 +3 \beta_2-3m,m,-\beta_2 + m)|\; m\in \NN\}.$ It follows from Theorem \ref{Fer-10-theorem} that $\phi_1$ generates the space of Gevrey solutions of $H_A(\beta)$ along $Y_\tau =\{x_3=0\}$ at any point $p\in \{x_3=0; x_1 x_2\neq 0\}$.
  
  
In particular, the space $\cN_w (H_A (\beta))$ is not contained in the space of Gevrey solutions of $M_A(\beta)$ along $Y_\tau$, although all the assumptions in Theorem \ref{thm: Gevrey-equal-Nilsson} except the condition $\pos(A)=\pos(A_\tau)$ are satisfied.
\end{example}

\setlength{\unitlength}{10mm}
$$\begin{picture}(-7,5)(10,2)
\put(5,4){\vector(0,1){3}}
\put(4,5){\vector(1,0){5}}
\put(6,5){\line(-1,1){1}}
\put(6,5){\line(2,-1){2}}
\put(5,5){\line(3,-1){3}}
\put(6,5){\makebox(0,0){$\bullet$}}
\put(5,6){\makebox(0,0){$\bullet$}}
\put(8,4){\makebox(0,0){$\bullet$}}
\put(6,6){\makebox(-0.7,-0.7){$\tau$}}
\put(7,5){\makebox(0.1,-0.7){$\sigma$}}
\put(6,5){\makebox(0.6,0.4){$a_1$}}
\put(5,6){\makebox(0.7,0.2){$a_2$}}
\put(8,4){\makebox(0.8,-0.2){$a_3$}} 
\put(6,3){\makebox(0,0){Figure 1}}
\end{picture}$$

\raggedbottom
\def\cprime{$'$} \def\cprime{$'$}
\providecommand{\MR}{\relax\ifhmode\unskip\space\fi MR }
\providecommand{\MRhref}[2]{%
  \href{http://www.ams.org/mathscinet-getitem?mr=#1}{#2}
}
\providecommand{\href}[2]{#2}

\end{document}